\documentclass[12pt]{amsart}
\usepackage{amsmath}
\usepackage{amsmath, pb-diagram}
\usepackage[all,cmtip]{xy}
\usepackage{amssymb}
\usepackage{amscd}
\newtheorem{thm}{Theorem}[section]
 \newtheorem{cor}[thm]{Corollary}
 \newtheorem{lem}[thm]{Lemma}
 \newtheorem{prop}[thm]{Proposition}
 \theoremstyle{definition}
 
  \newtheorem{ex}[thm]{Example}
  \newtheorem{fact}[thm]{Fact}
 \theoremstyle{remark}
 
 \numberwithin{equation}{section}

\DeclareMathOperator{\Z}{\mathbb{Z}}
\DeclareMathOperator{\End}{End}
\DeclareMathOperator{\Hom}{Hom}
\DeclareMathOperator{\Ker}{Ker}
\DeclareMathOperator{\im}{Im}
\DeclareMathOperator{\Soc}{Soc}

\begin{document}
\author[Ko\c{s}an, Quynh, Srivastava]{ M. Tamer Ko\c{s}an, Truong Cong Quynh and ASHISH K. SRIVASTAVA}
\address{Department of Mathematics, Gebze Technical University, 41400 Gebze/Kocaeli, Turkey}
\email{mtkosan@gyte.edu.tr \ tkosan@gmail.com}
\address{Department of Mathematics, Danang University, 459 Ton Duc Thang, DaNang city, Vietnam}
\email{tcquynh@dce.udn.vn; tcquynh@live.com}
\address{Department of Mathematics and Computer Science, St. Louis University, St. Louis,
MO-63103, USA}
\email{asrivas3@slu.edu}
\title{Rings with each right ideal automorphism-invariant}
\thanks{}

\maketitle

\begin{abstract} In this paper, we study rings having the property that every right ideal is automorphism-invariant. Such rings are called right $a$-rings. It is shown that (1) a right $a$-ring is a direct sum of a square-full semisimple artinian ring and a right square-free ring, (2) a ring $R$ is semisimple artinian if and only if the matrix ring $\mathbb{M}_n(R)$ is a right $a$-ring for some $n>1$, (3) every right $a$-ring is stably-finite, (4) a right $a$-ring is von Neumann regular if and only if it is semiprime, and (5) a prime right $a$-ring is simple artinian. We also describe the structure of an indecomposable right artinian right non-singular right
$a$-ring as a triangular matrix ring of certain block matrices.
\end{abstract}

\section{Introduction}

\noindent The study of rings characterized by homological properties of their one-sided ideals has been an active area of research. Rings for which every right ideal is quasi-injective (known as right $q$-rings) were introduced by Jain, Mohamed and Singh in \cite{JMS} and have been studied in a number of other papers (\cite{BJ}, \cite{BJ1}, \cite{Byrd}, \cite{Hill}-\cite{Koeh2}, \cite{Mohe1} and \cite{Mohe2}) by Beidar, Byrd, Hill, Ivanov and Koehler among other people. In \cite{Jain2} Jain, Singh and Srivastava studied rings whose each right ideal is a finite direct sum of quasi-injective right ideals and called such rings right $\Sigma$-$q$ rings. Jain, L\'opez-Permouth and Syed in \cite{Jain1} studied rings with each right ideal quasi-continuous and in \cite{CH} Clark and Huynh studied rings with each right ideal, a direct sum of quasi-continuous right ideals.

Recall that a module $M$ is called quasi-injective if $M$ is invariant under any endomorphism of its injective envelope; equivalently, any homomorphism from a submodule of $M$ to $M$ extends to an endomorphism of $M$. As a natural generalization of these modules, Dickson and Fuller \cite{DF} initiated the study of modules which are invariant under any automorphism of their injective envelope. These modules have been recently named as automorphism-invariant modules by Lee and Zhou in \cite{Lee-Zhou}. In \cite{ESS} Er, Singh and Srivastava proved that a module $M$ is automorphism-invariant if and only if any monomorphism from a submodule of $M$ to $M$ extends to an endomorphism of $M$ thus showing that automorphism-invariant modules are precisely the pseudo-injective modules studied by Jain and Singh in \cite{JS} and Teply in \cite{Teply}. Guil Asensio and Srivastava \cite{GS} have shown that automorphism-invariant modules satisfy the full exchange property and these modules also provide a new class of clean modules. The decomposition of automorphism-invariant  modules has been described in \cite{ESS}. If $M$ is an automorphism-invariant module, then $M$ has a decomposition $M=A\oplus B$ where $A$ is quasi-injective and $B$ is square-free. Recall that a module $M$ is called {\it square-free} if $M$ does not contain a nonzero submodule $N$ isomorphic to $X\oplus X$ for some module $X$. Recently, Guil Asensio, Keskin T\"ut\"unc\"u and Srivastava \cite{GKS} have initiated the study of a more general theory of modules invariant under automorphisms of their covers and envelopes. See \cite{AFT}, \cite {AS1}, \cite{AS2}, \cite{QK} and \cite{SS} for more details on automorphism-invariant modules.

Rings all of whose right ideals are automorphism-invariant are called {\it right $a$-rings} \cite{SS}. Since every quasi-injective module is automorphism-invariant, the family of right $a$-rings includes right $q$-rings. In fact, the class of right $a$-rings is a much larger class than the class of right $q$-rings as there exist examples of rings that are right $a$-rings but not right $q$-rings. The goal of this paper is to study these right $a$-rings and to extend the results in \cite{JMS} for this new class of rings. In particular, we show that:

\noindent
(1) A right $a$-ring is a direct sum of a square-full semisimple artinian ring and a right square-free ring (Theorem \ref{thm:g});

\smallskip

\noindent (2) A ring $R$ is semisimple artinian if and only if the matrix ring $\mathbb{M}_n(R)$ for some $n>1$ is an $a$-ring (Theorem \ref{pro:ma});

\smallskip

\noindent (3) If $R$ is a right $a$-ring, then $R$ is stably-finite, that is, every matrix ring over $R$ is directly-finite (Theorem \ref{sf});

\smallskip

\noindent (4)  A right $a$-ring is von Neumann regular if and only if it is semiprime (Theorem \ref{reg}), and a prime right $a$-ring is simple artinian (Theorem \ref{prime}).

\smallskip

\noindent We also characterize  indecomposable non-local right CS right $a$-rings. It is  shown that:

\smallskip
\noindent (5) Let $R$ be an indecomposable, non-local ring. Then  $R$ is a right $q$-ring if and only if $R$ is right CS and a right $a$-ring (Theorem \ref{baba}).

Let  $\Delta$ be a right $q$-ring with an essential maximal right ideal $P$ such that $\Delta/P$ is an injective right $\Delta$-module. In a right $q$-ring, every essential right ideal is two-sided by \cite[Theorem 2.3]{JMS}. Hence $\Delta/P$ is a division ring. Let $n$ be an integer with $n\geq 1$, let $D_1, D_2,\dots,D_n$ be division rings and $\Delta$ be a right $q$-ring, all of whose idempotents are central  and the right $\Delta$-module $\Delta/P$ is not embeddable into $\Delta_{\Delta}$. Next, let $V_i$ be $D_i$-$D_{i+1}$-bimodule such that
$$dim(\{V_i\}_{D_{i+1}})=1$$ for all $i=1,2,\dots,n-1$,  and let $V_n$ be a $D_n$-$\Delta$-bimodule such that $V_nP=0$ and $$dim(\{V_n\}_{\Delta/P})=1.$$ We denote
by $G_n(D_1,\dots,D_n, \Delta, V_1,\dots, V_n)$, the ring of $(n+1)\times (n+1)$ matrices of the form
$$
G_n(D_1,\dots,D_n, \Delta, V_1,\dots, V_n):=\begin{pmatrix}
D_1&V_1&&&&  \\
&D_2&V_2&&&\\
&&D_3&V_3&&\\
&&.&.&.&\\
&&&&.&.&\\
&&&&.&D_n&V_n\\
&&&&&&\Delta\\
\end{pmatrix}.$$ Consider the ring $G(D, \Delta, V)$. In \cite[Theorem 4.1]{BJ1}, it is shown that  $G(D, \Delta, V)$ is a right $q$-ring. Note that if we consider the transpose then it is a left $q$-ring. In the present paper, we obtain that

(6) $ G_n(D_1,\dots,D_n, \Delta, V_1,\dots, V_n)$ is a right $a$-ring  all of whose idempotents are central, where $\Delta$ is a right $a$-ring, $dim ({}_{D_i}\{V_i\})=dim(\{V_i\}_{D_{i+1}})=1$ for all $i=1,2,\dots,n-1$ and $dim ({}_{D_n}\{V_n\})=dim(\{V_n\}_{\Delta/P})=1$ (Theorem  \ref{ornn}).

Finally, we finish our paper with a structure theorem for an indecomposable right artinian right non-singular right
$a$-ring as a triangular matrix ring of certain block matrices.

Throughout this article all rings are associative rings with identity and all modules are right unital unless stated otherwise. For a submodule $N$ of $M$, we use $N \leq M$ ($N<M$)  to mean that $N$ is a submodule of $M$ (respectively,  proper submodule), and we write $N\leq^e M$ and $N\leq^{\oplus}  M$ to indicate that $N$ is an essential submodule of $M$ and $N$ is a direct summand  of $M$, respectively. We denote by $\Soc(M)$ and $E(M)$, the socle and the injective envelope of $M$, respectively. For any term not defined here the reader is referred to \cite{AF}, \cite{JST} and \cite{MM}.

\bigskip

\section{An Example}

\noindent As already mentioned, any right $q$-ring is a right $a$-ring. Recall that right $q$-rings are precisely those right self-injective rings for which every essential right ideal is a two-sided ideal \cite{JMS}. So, in particular, any commutative self-injective ring is a $q$-ring and hence an $a$-ring. Now we would like to present some examples of right $a$-rings that are not right $q$-rings. First, we have the following useful observation.

\begin{lem}
A commutative ring is an $a$-ring if and only if it is an automorphism-invariant ring.
\end{lem}

\begin{proof}
Let $R$ be a commutative automorphism-invariant ring and $I$ be an ideal of $R$. There exists an ideal $U$ of $R$ such that $I\oplus U$ is essential in $R$. Then $E(R)=E(I\oplus U)$. Let $\varphi$ be an automorphism of $E(R)$. Clearly, $\varphi(1)\in R$. Now, for all $x\in I\oplus U$, we have  $\varphi(x)=\varphi(1)x\in I\oplus U$. So $\varphi(I\oplus U)\leq I\oplus U$ which implies that $I\oplus U$ is an automorphism-invariant module. Since direct summand of an automorphism-invariant module is automorphism-invariant, it follows that $I$ is automorphism-invariant. This shows that $R$ is an $a$-ring. The converse is obvious.
\end{proof}

In view of the above, we have the following example of $a$-ring which is not a $q$-ring.

\begin{ex}
Consider the ring $R$ consisting of all eventually constant sequences of elements from $\mathbb F_2$ (see {\cite[Example 9]{ESS}}). Clearly, $R$ is a commutative automorphism-invariant ring as the only automorphism of its injective envelope is the identity automorphism. Hence $R$ is an $a$-ring by the above lemma. But $R$ is not a $q$-ring because $R$ is not self-injective.
\end{ex}

\bigskip

\section{Some characterizations of $a$-rings}

\noindent In this section we will prove some characterizations for right $a$-rings. These equivalent characterizations will be more convenient to use.

\begin{prop}\label{essen} The following conditions are equivalent for a ring $R$:
\begin{enumerate}
\item $R$ is a right $a$-ring.
\item Every essential right ideal of $R$ is automorphism-invariant.
\item $R$ is right automorphism-invariant and every essential right ideal of $R$ is a left $T$-module, where  $T$ is a subring of $R$ generated by its unit elements.
\end{enumerate}
\end{prop}
\begin{proof} $(1)\Rightarrow (2)$. This is obvious.

$(2)\Rightarrow (3)$. By the hypothesis, $R$ is a right automorphism-invariant ring. Let $I$ be an essential right ideal of $R$. Then $E(I)=E(R)$. Let $T$ be a subring of $R$ generated by its units. Then $T$ is  a subring of $\End(E(R))$, and so $TI=I$.

$(3)\Rightarrow (2)$. Let $I$ be an essential right ideal of $R$. Then $E(I)=E(R)$. Let $\varphi$ be an automorphism of $E(R)$. As $R$ is right automorphism-invariant, we have $\varphi(R)= R$ which implies that $\varphi(1)$ is a unit of $R$. By the assumption, we have $\varphi(1)I\leq I$ and so $\varphi(I)\leq I$. This shows that each essential right ideal of $R$ is automorphism-invariant.

$(2)\Rightarrow (1)$. Let $A$ be any right of $R$. Let $K$ be a complement of $A$, then $A\oplus K$ is an essential right ideal of $R$. By assumption, $A\oplus K$ is automorphism-invariant. Since every direct summand of an automorphism-invariant module is automorphism-invariant, it follows that $A$ is automorphism-invariant. This proves that $R$ is a right $a$-ring. 
\end{proof}

\begin{cor} Let $R=S\times T$ be a product of rings. Then $R$ is a right $a$-ring if and only if $S$ and $T$ are right $a$-rings.
\end{cor}

The singular submodule $Z(M)$ of a right $R$-module $M$ is defined as $Z(M)=\{m\in M: ann^{r}_R(m)$ is an essential right ideal of $R\}$ where $ann^r_R(m)$ denotes the right annihilator of $m$ in $R$.
The singular submodule of $R_R$ is called the (right) singular
ideal of the ring $R$ and is denoted by $Z(R_R)$. It is well known that $Z(R_R)$ is
indeed an ideal of $R$.

\begin{lem}\label{key}Let $R$ be a right $a$-ring and $A$ be a right ideal of $R$. If there exists a right ideal $B$ of $R$ with $A\cap B=0$ and $A\cong B$, then:
\begin{enumerate}
\item $A$ is semisimple and injective.
\item $A$ is nonsingular.
\end{enumerate}
\end{lem}
\begin{proof} (1) Let $A$ and $B$ be right ideals of a right $a$-ring $R$ with $A\cap B=0$ and $A\cong B$. Let $D$ be a complement of $A\oplus B$ in $R_R$. Then $(A\oplus B)\oplus D\leq ^e R_R$. It follows that $E((A\oplus B)\oplus D)\leq ^e E(R_R)$. On the other hand,  $E((A\oplus B)\oplus D)$ is a direct summand of $E(R_R)$ and so $E((A\oplus B)\oplus D)=E(R_R)$. We have $E((A\oplus B)\oplus D)=E(A)\oplus E(B)\oplus E(D)$. Thus $E(R_R)=E(A)\oplus E(B)\oplus E(D)$ which means that we have a decomposition $E(R_R)=E(A)\oplus E(B)\oplus C$  for some $C\leq E(R_R)$. Note that $E(A)\cong E(B)$ and $R$ is right automorphism-invariant. By \cite[Lemma 7]{SS}, we get  $$R_R=(R\cap E(A))\oplus (R\cap E(B))\oplus (R\cap C).$$ We also have $B\cap (R\cap E(A))=0$ and $A\cap [(R\cap E(B))\oplus (R\cap C)]=0$. Since $R$ is a right $a$-ring, the modules $B\oplus [R\cap E(A)]$ and $A\oplus [(R\cap E(B))\oplus (R\cap C)]$ are automorphism-invariant. By \cite[Theorem 5]{Lee-Zhou},  $B$ is $[R\cap E(A)]$-injective and $A$ is $[(R\cap E(B))\oplus (R\cap C)]$-injective . Note that $A\cong B$. Thus $A$ is $R$-injective, that is, $A$ is injective. Let $\varphi: A\to B$ be an isomorphism and $U$ be a submodule of $A$. Clearly, $U\cong \varphi(U)$. Let $V=\varphi(U)$. Then $U\cap V=0$ and $U\cong V$. By a similar argument as above, we have that $U$ is an injective module and thus it follows that $U$ is a direct summand of $A$. This proves that $A$ is semisimple.

(2) Let $a$ be an arbitrary element of $Z(A)$. Then $aR$ is an injective module since it is a direct summand of $A$. It follows that $aR=eR$ for some $e^2=e\in R$. Therefore $e\in Z(A)$ and so $e=0$. Thus $a=0$ which shows  $Z(A)=0$.
\end{proof}

\noindent Recall that two modules $M$ and $N$ are said to be {\it orthogonal} if no submodule of $M$ is isomorphic to a submodule of $N$. A module $M$ is said to be a {\it square module} if there exists a right module $N$ such that  $M\cong N^2$ and a submodule $N$ of a module $M$ is called {\it square-root} in $M$ if $N^2$ can be embedded in $M$. A module $M$ is called {\it square-free} if $M$ contains no non-zero  square roots and $M$ is called {\it square-full} if every submodule of $M$ contains a non-zero square root in $M$.

As a consequence of the above lemma, we are now ready to prove a useful decomposition theorem for any right $a$-ring.

\begin{thm}\label{thm:g} A right $a$-ring is a direct sum of a square-full semisimple artinian
ring and a right square-free ring.
\end{thm}
\begin{proof} By \cite[Theorem 3]{ESS}, there exists a decomposition $R_R=A\oplus B\oplus C$ where $A\cong B$ and the module $C$ is square-free which is orthogonal to  $A\oplus B$. Let $X:=A\oplus B$ and $Y:=C$. Now we proceed to show that $X$ is square-full. Let $U$ be a non-zero arbitrary submodule of $X$. There exist either non-zero submodules $U_1$ of $U$ and $V_1$ of $A$ such that $U_1\cong V_1$ or non-zero submodules  $U_2$ of $U$ and  $V_2$ of $B$ such that $U_2\cong V_2$. It follows that  $U_1^2$ or $U_2^2$ can be embedded in $X$. This means $U$ contains a square root in $X$ and hence $X$ is square-full.

By Lemma \ref{key}, $A$ and $B$ are injective semisimple modules and so $X$ is injective and semisimple. Next we show that $X$ and $Y$ are ideals of $R$. Let $f$ be a nonzero homomorphism from $X$ to $Y$. As $X$ is semisimple, $\Ker(f)$ is a direct summand of $X$. So there exists a submodule $L$ of $X$ such that $X=\Ker(f)\oplus L$. As $\Ker(f)\cap L=0$, we have $L\cong f(L) \subseteq Y$, a contradiction to the fact that $X$ is orthogonal to $Y$. Hence, we have $\Hom(X, Y)=0$. Assume that $\varphi: Y\to X$ is a non-zero homomorphism. Then $Y/\Ker(\varphi)\cong \im(\varphi)$ is projective (since $\im(\varphi)$ is a direct summand of $X$). It follows that there exists a non-zero submodule $K$ of $Y$ such that $\Ker(\varphi)\cap K=0$. So $K\cong \varphi(K)$, a contradiction to the orthogonality of $X$ and $Y$. Therefore $\Hom(Y, X)=0$.

Thus $R=X\oplus Y$, where $X$ is a square-full semisimple artinian
ring and $Y$ is a right square-free ring.
\end{proof}

As a consequence of the above, we have

\begin{cor} 
An indecomposable ring $R$ containing a square is a right $a$-ring if and only if $R$ is simple artinian.
\end{cor}

We denote the ring of $n\times n$ matrices over a ring $R$ by $\mathbb M_n(R)$. In the next theorem we study when matrix rings are right $a$-rings.

\begin{thm}\label{pro:ma} Let $n>1$ be an integer. The following conditions are equivalent for a ring $R$:
\begin{enumerate}
\item $\mathbb{M}_n(R)$ is a right $q$-ring.
\item $\mathbb{M}_n(R)$ is a right $a$-ring.
\item $R$ is semisimple artinian.
\end{enumerate}
\end{thm}
\begin{proof} The implication $(1)\Rightarrow (2)$ is obvious.

\smallskip

$(2)\Rightarrow (3)$. Let $\mathbb{M}_n(R)$ be a right $a$-ring. Assume that $R$ is not semisimple artinian. Then there exists an essential right ideal, say $B$, of $R$ such that $B\neq R$. Define $E:=\{\sum a_{ij}e_{ij}:  a_{1j} \in B, 1\leq j\leq  n $ and $ a_{ij}\in R, 1\leq i,j\leq n \}$ where $e_{ij}$ ($1\leq i,j\leq n$) are the units of $\mathbb{M}_n(R)$. Then clearly $E$ is an essential right ideal of $\mathbb{M}_n(R)$. Consider the unit
$\begin{pmatrix} 0&0&0& \cdots &0 &1\\ 0&0& 0& \cdots &1&0\\ 0 & 0& 0& \cdots &0 &0 \\ \vdots & \vdots & \vdots & \ddots & \vdots & \vdots \\1&0&0& \cdots &0 &0\\  \end{pmatrix}$  of $\mathbb{M}_n(R)$. Then
$$\begin{pmatrix} 0&0&0& \cdots &0 &1\\ 0&0& 0& \cdots &1&0\\ 0 & 0& 0& \cdots &0 &0 \\ \vdots & \vdots & \vdots & \ddots & \vdots & \vdots \\1&0&0& \cdots &0 &0\\  \end{pmatrix}\begin{pmatrix} 0&0&0& \cdots &0 &0\\ 0&0& 0& \cdots &0&0\\ 0 & 0& 0& \cdots &0 &0 \\ \vdots & \vdots & \vdots & \ddots & \vdots & \vdots \\0&0&0& \cdots &0 &1\\  \end{pmatrix}=\begin{pmatrix} 0&0&0& \cdots &0 &1\\ 0&0& 0& \cdots &0&0\\ 0 & 0& 0& \cdots &0 &0 \\ \vdots & \vdots & \vdots & \ddots & \vdots & \vdots \\0&0&0& \cdots &0 &0\\  \end{pmatrix}\not \in E.$$ This yields a contradiction (see Proposition \ref{essen}). Hence, $R$ is semisimple artinian.

\smallskip

$(3)\Rightarrow (1)$. This is obvious.
\end{proof}

The following example shows that there exists automorphism-invariant rings which are not right $a$-rings.

\begin{ex}Let $R=\Z_{p^n}$, where $p$ is a prime and $n>1$. It is well known that  $R$ is self-injective. By \cite[Theorem 8.3]{Ut}, $\mathbb{M}_m(R)$ is right self-injective for all $m> 1$. Thus, for instance, $\mathbb{M}_m(\Z_{p^2})$ is a right automorphism-invariant ring. But $\mathbb{M}_m(\Z_{p^2})$ is not a right $a$-ring for any $m>1$ in the view of above theorem as $\Z_{p^2}$ is not semisimple artinian. This example also shows that being a right $a$-ring is not a Morita invariant property.
\end{ex}

\bigskip

\section{Special classes of right $a$-rings}

\noindent In this section, we will consider some special classes of rings, for example, simple, semiprime, prime and CS and characterize as to when these rings are right $a$-rings. We begin this section with a simple observation.

\begin{lem}\label{lem:2.3}Let $A$ and $B$ be right ideals of a  right $a$-ring $R$ with $A\cap B=0$.  Then the following conditions  hold:
\begin{enumerate}
 \item  If $\varphi: A\to B$ is a nonzero homomorphism, then
    \begin{enumerate}
        \item [$(i)$] $\varphi(A)$ is a semisimple module.
        \item [$(ii)$]  $\varphi(A)$ is simple if $B$ is uniform.
    \end{enumerate}
\item If $e$ is a non-trivial idempotent of $R$ such that $eR(1-e)\ne 0$, then $\Soc(eR)\ne 0$.
\end{enumerate}
\end{lem}

\begin{proof}
$(1) (i)$.  Let $U$ be an arbitrary essential submodule of $B$. Then $E(U)=E(B)$ and $U\oplus A$ is automorphism-invariant. It follows that $U$ is $A$-injective. On the other hand, there exists a homomorphism $\bar{\alpha}: E(A)\to E(B)$ such that $\bar{\alpha}|_A=\varphi$. It follows that $\bar{\alpha}(A)\leq U$ and so $\varphi(A)\leq U$. This shows that  $\varphi(A)\leq \Soc(B)$.

$(ii)$. If $B$ is uniform, then from (i), it follows easily that $\varphi(A)$ is simple.

$(2)$. Assume that $eR(1-e)\ne 0$. There exists $r_0\in R$ such that $er_0(1-e)\ne 0$. Consider the homomorphism  $\beta: (1-e)R\to eR$ defined by $\beta((1-e)x)=er_0(1-e)x$. Clearly, $\beta$ is well-defined and $\im(\beta)\ne 0$. By (1)(i), we have $\im(\beta)\leq \Soc(eR)$. Hence $\Soc(eR)\ne 0$.
\end{proof}

\bigskip
\noindent Recall that a ring $R$ is called {\it von Neumann regular} if for every $a\in R$, there
exists some $b\in R$ such that $a = aba$. A ring $R$ is said to be {\it prime} if the product of any two nonzero ideals of $R$ is nonzero and a ring $R$ is called {\it semiprime} if it has
no nonzero nilpotent ideals.

\begin{thm} \label{reg} A right $a$-ring is von Neumann regular if and only if it is semiprime.
\end{thm}
\begin{proof} Let $R$ be a right $a$-ring. If $R$ is von Neumann regular, then it is well known that $R$ is semiprime. Conversely, assume that $R$ is semiprime. As $R$ is a right $a$-ring, in particular, $R_R$ is automorphism-invariant. By \cite[Proposition 1]{GS}, $R/J(R)$ is von Neumann regular and $J(R)=Z(R_R)$. Now we proceed to show that $J(R)=0$. In fact, for any $x\in J(R)$, there exists an essential right ideal $E$ of $R$ such that $xE=0$. Since $R$ is a right $a$-ring, $uE\leq E$ for all units $u$ in $R$ by Lemma \ref{essen}. It follows that $(RxR)E\leq E$ and so $(xRxR)E\leq xE=0$, and so  either $xRxR\leq P$ or $E\leq P$ for all prime ideal $P$ of $R$. Let $\{P_i\}_{i\in I}$ and $\{P_j\}_{j\in J}$ be families of all prime ideals of $R$ such that $xRxR\leq P_i$ for all $i\in I$ and $xRxR\not \leq P_j$ for all $j\in J$. Taking $X=\cap_{i\in I}P_i$ and $Y=\cap_{j\in J}P_j$. Since $R$ is semiprime, $X\cap Y=0$. Moreover, we have $E\leq Y$ and so $Y\leq^e R_R$. If $xRxR\ne 0$, there exists $r_1,r_2\in R$ such that $xr_1xr_2\not= 0$. Then there exists a $y\in R$ such that $xr_1xr_2y\ne 0$ and $xr_1xr_2y\in Y$, a contradiction. Thus $xRxR=0$. Furthermore, as $R$ is semiprime, we have $x=0$. This completes the proof.
\end{proof}

Recall that a ring $R$ is called \textit{directly-finite} if $xy=1$ implies $yx=1$ for all $x,y\in R$. Assume that $R$ is a right $a$-ring. By Theorem \ref{thm:g}, we have a decomposition $R=S\times T$, where $S_S$ is semi-simple artinian and $T_T$ is square-free. Since $S$ and $T$ are directly-finite rings, one infers that the ring $R$ is also directly-finite. Next, we will see that a right $a$-ring is not only directly-finite but it is stably-finite. If for a ring $R$, every matrix ring $\mathbb M_n(R)$ is directly finite then $R$ is called a {\it stably-finite ring}. It is known that the property of stable-finiteness is of importance in topology as well as in the theory of operator algebras.

A ring $R$ is called \textit{right quasi-duo} (left quasi-duo) if every maximal right ideal (every maximal left ideal) is two-sided. The condition that maximal right ideals be two-sided first appeared in the work of Burgess and Stephenson \cite{BS}. This notion has a natural connection to left and right unimodular sequences too (see \cite{LD}). It is still an open problem whether quasi-duo rings are left-right symmetric or not.

\begin{thm} \label{sf}
Every right $a$-ring is stably-finite.
\end{thm}

\begin{proof}
Let $R$ be a right $a$-ring. Then $R=S\times T$, where $S_S$ is semi-simple artinian and $T_T$ is square-free. By \cite[Theorem 15]{AS2}, $T$ is a right quasi-duo ring\label{qduo}. Then $\mathbb M_n(R)=\mathbb M_n(S\oplus T)$. Thus $\mathbb M_n(R) \cong \mathbb M_n(S)\oplus \mathbb M_n(T)$. Clearly, $\mathbb M_n(S)$ is directly finite. Now we proceed to show that $\mathbb M_n(T)$ is directly finite. Let $\{M_i\}$ be the set of maximal right ideals of the quasi-duo ring $T$. Then each $M_i$ is a two-sided ideal and $J(T)=\cap M_i$. Clearly, each $T/M_i$ is a division ring. Thus $\mathbb M_n(T)/\mathbb M_n(M_i)\cong \mathbb M_n(T/M_i)$ is a simple artinian ring which is clearly directly  finite. Consider the natural ring homomorphism $\varphi: \mathbb M_n(T)\longrightarrow \prod_i{\mathbb M_n(T/M_i)}$. We have $\Ker(\varphi)=\mathbb M_n(J(T))=J(\mathbb M_n(T))$. Since each $\mathbb M_n(T/M_i)$ is directly finite, $\prod_i \mathbb M_n(T/M_i)$ is directly finite and consequently, $\mathbb M_n(T)/J(\mathbb M_n(T))$ is directly finite being a subring of a directly finite ring. Hence $\mathbb M_n(T)$ is directly finite. Thus $\mathbb M_n(R)$ is directly finite and therefore $R$ is stably-finite.
\end{proof}

A ring $R$ is called {\it unit-regular} if, for every element $x\in R$, there exists a unit $u \in  R$ such that $x = xux$. We can now have the following result.

\begin{cor} \label{vnr}
Every von Neumann regular right $a$-ring is unit-regular.
\end{cor}

\begin{cor} The ring of linear transformations $R:=\End(V_D)$ of a vector space $V$  over a division ring $D$  is a right $a$-ring if and only if the vector space is finite-dimensional.
\end{cor}

\begin{proof}
If $V$ is an infinite-dimensional vector space over $D$ then $\End(V_D)$ is not directly-finite. So the result follows from Theorem \ref{sf}.
\end{proof}

\noindent A ring $R$ is said to be {\it strongly regular} if for every $a\in R$,
there exists some $b\in R$ such that $a = a^2b$.

\begin{prop}Let $R$ be a semi-prime right $a$-ring with zero socle. Then $R$ is strongly
regular.
\end{prop}
\begin{proof} Assume that $R$ is a semi-prime right $a$-ring. Clearly, $R$ is von Neumann regular. Let $e$ be an idempotent in $R$. Suppose $(1-e)Re\ne 0$. Then $Soc((1-e)R)\ne 0$, a contradiction. Hence $(1-e)Re=0$ and this shows that $e$ is a central idempotent (see \cite[Lemma 2.33]{G2}). Because every idempotent of $R$ is central, $R$ is strongly regular.
\end{proof}

\begin{thm} \label{prime}
Let $R$ be a prime ring. Then $R$ is a right $a$-ring if and only if $R$ is a simple artinian ring.
\end{thm}
\begin{proof} Assume that $R$ is a prime right $a$-ring. In view of Theorem \ref{thm:g}, we obtain that either $R$ is a simple artinian ring or $R$ is a square-free ring. So, it suffices to consider the case that $R$ is a square-free prime right $a$-ring. By Theorem \ref{reg}, $R$ is a von Neumann regular ring. Since $R$ is square-free, all idempotents of $R$ are central and hence $R$ is a strongly regular ring. Now as every prime strongly regular ring is a division ring, the result follows.
\end{proof}

In particular, from the above theorem it follows that every simple right $a$-ring is artinian.
\bigskip

\noindent A module $M$ is called a {\it CS module} if every submodule of $M$ is essential in a direct summand of $M$ \cite{MM}. And a module $M$ is called a {\it weak CS module} if every semisimple submodule of $M$ is essential in a direct summand of $M$ \cite{Smith}. It is shown in \cite{ESS} that any automorphism-invariant module $M$ satisfies the C2 property. Now if we assume, in addition, that $M$ is a CS module, then $M$ is a continuous module and hence $M$ is invariant under any idempotent endomorphism of $E(M)$. Since $E(M)$ is a clean module, being an injective module, any endomorphism of $E(M)$ is a sum of an idempotent endomorphism and an automorphism. Thus, a CS automorphism-invariant module $M$ is invariant under any endomorphism of $E(M)$ and consequently it is quasi-injective. This means that if $R$ is a right CS, right automorphism-invariant ring then $R$ is right self-injective. Therefore, we consider a weaker condition in the next proposition.

\begin{prop} Let $R$ be a right weak CS right $a$-ring. If $e$ is a primitive idempotent of $R$ such that $eR(1-e)\ne 0$, then $eRe$ is a division ring and $eR(1-e)$ is the only proper $R$-submodule of $eR$.
\end{prop}
\begin{proof} By Lemma \ref{lem:2.3}, $\Soc(eR)\ne 0$. Since $R$ is right automorphism-invariant, $R$ is right C2 by \cite{ESS}. By \cite[Theorem 1.4]{Er}, $eR$ is also a weak CS module. Firstly, we show that $\Soc(eR)$ is a simple module which is essential in $eR$. Since $eR$ is a weak CS module, $\Soc(eR)$ is essential in a direct summand of $eR$. But $eR$ is an indecomposable module which implies that $\Soc(eR)$ is essential in $eR$. For any nonzero arbitrary element $a\in \Soc(eR)$, we obtain that $aR$ is essential in $eR$ (because $eR$ is an indecomposable  weak CS module). It follows that $\Soc(eR)\leq aR$ and so $\Soc(eR)= aR$. Thus $\Soc(eR)$ is a simple module. Therefore $eR$ is uniform. Since a uniform automorphism-invariant module is quasi-injective, $eR$ is quasi-injective. Thus $eRe\cong \End(eR)$ is a local ring, that is, $e$ is a local idempotent of $R$.

Next we show that $eR(1-e)$  is the only proper submodule of $eR$. Since $eR(1-e)\ne 0$, one infers $eR(1-e)\subset \Soc(eR)$ by Lemma \ref{lem:2.3}. Hence  $$eR(1-e)=\Soc(eR)(1-e).$$ We next show that $eJ(R)e$ is a submodule of $eR$. Since $R$ is right automorphism-invariant, $J(R)=Z(R_R)$ by \cite[Proposition 1]{GS} and so $J(R)\Soc(eR)=0$. Now  $(eJ(R)e)\Soc(eR)=  eJ(R)\Soc(eR)= 0$ and so $(eJ(R)e)(eR(1-e))=0$. On the other hand, we have  $$eJ(R)eR=  eJ(R)e(Re+R(1-e))=eJ(R)eRe\subset eJ(R)e.$$ Hence $eJ(R)e$ is an $R$-submodule of $eR$. Since $\Soc(eR)$ is simple, we have  $eJ(R)e\cap \Soc(eR)=0$ or $\Soc(eR)\leq eJ(R)e$. Suppose $\Soc(eR)\leq eJ(R)e$. Then $eR(1-e)=\Soc(eR)(1-e)\leq eJ(R)e(1-e)=0$, a contradiction. It follows that  $eJ(R)e\cap \Soc(eR)=0$. Thus $eJ(R)e=0$.

Let $I$ be a proper submodule of $eR$. Since $eR$ is local,  $I\leq eJ(R)$ and so $Ie=0$. On the other hand, we have $I(1-e)\leq eR(1-e)$ which implies that $I\leq eR(1-e)=\Soc(eR)$. Thus $I=0$ or $I=\Soc(eR)$. In particular, we have $\Soc(eR)e=0$. Therefore $eR(1-e)=\Soc(eR)(1-e)=\Soc(eR)$.
\end{proof}

As a consequence, we have the following.

\begin{thm} \label{baba} Let $R$ be an indecomposable, non-local ring. The following conditions are equivalent:
\begin{enumerate}
\item $R$ is a right $q$-ring.
\item $R$ is a right CS right $a$-ring.
\end{enumerate}
\end{thm}

\begin{proof} This follows from previous proposition and \cite[Theorem 3]{I}.
\end{proof}

\section{Structure theorems}

\noindent In this section we would like to describe the structure of right $a$-rings. In the case of right $q$-rings, Byrd \cite{Byrd} and Ivanov (\cite{I}, \cite{I1}) gave a description
of right $q$-rings but their characterizations turned out to be not complete. Finally,
the structure of right $q$-rings was completely described by Beidar et al in \cite{BJ1}.

\begin{thm} $($Beidar, Fong, Ke, Jain, \cite{BJ1}$)$
A right $q$-ring $R$ is isomorphic to a finite direct product of right $q$-rings of the following types:
\begin{enumerate}
\item Semisimple artinian ring.
\item $H(n; D; id_D)$ where $id_D$ is the identity automorphism on division ring $D$.
\item $G(n; \Delta; P)$ where $\Delta$ is a right $q$-ring whose all
idempotents are central.
\item A right $q$-ring whose all idempotents are central.

\end{enumerate}

Here\\

$H(n; D; \alpha)= \left[
\begin{array}{ccccccc}
D & V & 0 & &  &  & 0 \\
0 & D & V & 0 & &  & 0 \\
&  & D & V & 0 &  &  \\
&  &  &  &  &  &  \\
&  &  &  & D & V & 0 \\
 &  &  &  &  & D & V \\
V(\alpha ) & 0 & &  &  &  & D%
\end{array}
\right] $, where $V$ is one-dimensional both as a left $D$-space and a right
$D$-space, $V(\alpha )$ is also a one-dimensional left $D$-space as well as
a right $D$-space with right scalar multiplication twisted by an
automorphism $\alpha $ of $D$, i.e., $vd = v\cdot\alpha(d)$ for all $v\in V$, $d\in D$,
\newline and
$$
G_n(n; \Delta; P):=\begin{pmatrix}
D&V&&&&  \\
&D&V&&&\\
&&D&V&&\\
&&.&.&.&\\
&&&&.&.&\\
&&&&.&D&V\\
&&&&&&\Delta\\
\end{pmatrix},$$
where $V$ is as above and $\Delta$ is a right $q$-ring with maximal essential right
ideal $P$ and hence $D=\Delta/P$ is a division ring.
\end{thm}

Now, using the above defined notations, we give the following description of right $a$-rings.

\begin{thm} \label{ornn} Let $n\ge 1$ be an integer, $D_1, D_2,\dots,D_n$ be division rings and $\Delta$ be a right $a$-ring with all idempotents central and an essential
ideal, say $P$, such that $\Delta/P$ is a division ring and the right $\Delta$-module $\Delta/P$ is not embeddable into $\Delta_{\Delta}$. Next, let $V_i$ be a $D_i$-$D_{i+1}$-bimodule such that
$$dim ({}_{D_i}\{V_i\})=dim(\{V_i\}_{D_{i+1}})=1$$ for all $i=1,2,\dots,n-1$,  and let $V_n$ be a $D_n$-$\Delta$-bimodule such that $V_nP=0$ and $$dim ({}_{D_n}\{V_n\})=dim(\{V_n\}_{\Delta/P})=1.$$
Then $R:= G_n(D_1,\dots,D_n, \Delta, V_1,\dots, V_n)$ is a right $a$-ring.
\end {thm}
\begin{proof} Let $1\leq i\leq n+1$ and  $e_i$ be the matrix whose $(i,i)$-entry is equal to $1$
and all the other entries are $0$. It is easy to see that $e_jRe_{j+1}$ are minimal right ideals of $R$ for all $j=1, 2,\dots, n$. Let $K$ be a right ideal of the ring $\Delta$ and let $\widehat{K}$ to be the set of all matrices whose $(n+1, n+1)$-entries are from $K$ and all the other entries are $0$. Given $1\leq i\leq n$ and a right ideal $K$ of $\Delta$. We are going to adapt the techniques of \cite[Proposition 2.16]{BJ}. First, we have the following facts which will be used throughout in the proof:
\begin{fact}\label{fa1} $e_iR\ \text{and}\ \widehat{K}$ are relatively injective. Also, $e_iRe_{i+1}$ and $\widehat{K}$ are relatively injective.
\end{fact}
\begin{fact}\label{fa2} $\Hom(e_iR, \widehat{K})= 0= \Hom(e_iRe_{i+1}, \widehat{K})$.
\end{fact}
\begin{fact} \label{fa3} $e_iR$ and $e_jR$ are relatively injective for all $j\ne i$. Also, $e_iRe_{i+1}$ and $e_jR$ are relatively injective for all $j\ne i$.
\end{fact}
Let $U$ be an essential right ideal of $R$. Then $e_iRe_{i+1}\leq U$ for
all $i=1, 2,\dots, n$. Set $W:=\sum_{i=1}^ne_iRe_{i+1}$. Note that $W$ is an ideal of $R$ and $W\leq U$.  Since the factor ring $R/W$ is isomorphic to the ring $(\oplus_{i=1}^nD_i)\oplus \Delta$ and $U/W$ is a right ideal of $R/W$,
we conclude that there exists a partition $I, J$ of the set $\{1, 2, \dots ,n\}$ and a right ideal $K$ of $\Delta$ such that $ U=(\oplus_{i\in I}e_iR)\oplus (\oplus_{j\in J}e_jRe_{j+1})\oplus \widehat{K}.$

Now we deduce the following useful conclusions.

(i) $\oplus_{j\in J}e_jRe_{j+1}$ is a semisimple right $R$-module and so $\oplus_{j\in J}e_jRe_{j+1}$ is quasi-injective.

(ii) $\oplus_{i\in I}e_iR$ is a  quasi-injective right $R$-module. In fact, by Fact \ref{fa3}, we only need to prove that each $e_iR$ is a quasi-injective right $R$-module for all $i\in I$. Note that $e_iRe_{i+1}$ is only proper submodule of $e_iR$. Let $f: e_iRe_{i+1}\to e_iR$ be an $R$-homomorphism. Note that $e_iRe_{i+1}=\begin{pmatrix}
0&0&&&&  \\
&0&0&&&\\
&&0&0&&\\
&&.&.&V_i&\\
&&&&.&.&\\
&&&&.&0&0\\
&&&&&&0\\
\end{pmatrix}.$ Then $f(e_iRe_{i+1})=e_iRe_{i+1}$.  Since $dim ({}_{D_i}\{V_i\})=dim(\{V_i\}_{D_{i+1}})=1$, there exists $v_i\in V_i$ such that $D_iv_i=v_iD_{i+1}$. Assume that $f(v_i)=\begin{pmatrix}
0&0&&&&  \\
&0&0&&&\\
&&0&0&&\\
&&.&.&v_id_{i+1}&\\
&&&&.&.&\\
&&&&.&0&0\\
&&&&&&0\\
\end{pmatrix}$ for some $d_{i+1}\in D_{i+1}$. There exists $d_i\in D_i$ such that $d_iv_i=v_id_{i+1}$. We consider the $R$-homomorphism $\bar{f}: e_iR\to e_iR$ defined as left multiplication by $\begin{pmatrix}
0&0&&&&  \\
&0&0&&&\\
&&0&0&&\\
&&.&d_i&0&\\
&&&&.&.&\\
&&&&.&0&0\\
&&&&&&0\\
\end{pmatrix} $. Then $\bar{f}$ is an extension of $f$. In the case of $e_nR$, it is similar.

(iii) $\widehat{K}=\widehat{K_1}\oplus \widehat{K_2}$, where  $\widehat{K_1}$ is a quasi-injective $R$-module and $\widehat{K_2}$ is a square-free automorphism-invariant $R$-module. In fact, by Theorem \ref{thm:g}, we have a decomposition $\Delta=\Delta_1\times \Delta_2$, where  $\Delta_1$ is semisimple artinian and $\Delta_2$ is square-free. It follows that there exists a quasi-injective $\Delta$-module $K_1$  and a square-free $\Delta$-module $K_2$ such that $K=K_1\oplus K_2$. Thus $\widehat{K}=\widehat{K_1}\oplus \widehat{K_2}$. Since $e_{n+1}R(1-e_{n+1})=0$, we obtain that $\widehat{K_1}$ is quasi-injective  and $\widehat{K_2}$ is square-free  by \cite[Lemma 2.3(6)]{BJ}. Furthermore, by the hypothesis, $\widehat{K_2}$ is automorphism-invariant.

Let $X=(\oplus_{i\in I}e_iR)\oplus (\oplus_{j\in J}e_jRe_{j+1})\oplus \widehat{K_1}$ and $Y=\widehat{K_2}$. Then $U=X\oplus Y$. By Facts \ref{fa1}, \ref{fa2} and \ref{fa3}, $X$ is quasi-injective, $Y$ is automorphism-invariant square-free which is orthogonal to $X$, and $X$ and $Y$ are relatively injective. By \cite{QK}, $U$ is automorphism-invariant. This shows that each essential right ideal of $R$ is automorphism-invariant. Now, let $A$ be any right ideal of $R$. Let $C$ be a complement of $A$ in $R$. Then $A\oplus C$ is an essential right ideal of $R$. Thus, as shown above, $A\oplus C$ is automorphism-invariant and consequently, $A$ is automorphism-invariant. This proves that $R$ is a right $a$-ring.
\end{proof}

We finish this paper by giving by another structure theorem for indecomposable right artinian right non-singular right $a$-ring describing them as a triangular matrix ring of certain block matrices.

\begin{thm} \label{ornn1}
 Any indecomposable right artinian right nonsingular right weakly CS right $a$-ring $R$ is isomorphic to
 $$ \begin{pmatrix} \mathbb{M}_{n_1}(e_1Re_1)&\mathbb{M}_{n_1\times n_2}(e_1Re_2)&\mathbb{M}_{n_1\times n_3}(e_1Re_3)&\cdots&\mathbb{M}_{n_1\times n_k}(e_1Re_k)\\
0&\mathbb{M}_{n_2}(e_2Re_2)&\mathbb{M}_{n_2\times n_3}(e_1Re_2)&\cdots &\mathbb{M}_{n_2\times n_k}(e_2Re_k)\\
0&0&.& \cdots & .\\
\vdots &\vdots& \vdots & \vdots & \vdots \\
0&0&0&\cdots&\mathbb{M}_{n_k}(e_kRe_k)
\end{pmatrix},
$$
where $e_iRe_i$ is  a division ring, $e_iRe_i\simeq e_jRe_j$ for each $1\leq i,j\leq k$ and $n_1, \ldots, n_k $ are any positive integers.  Furthermore, if $e_iRe_j\ne 0$, then $$dim(_{e_iRe_i}(e_iRe_j))=1=dim((e_iRe_j)_{e_jRe_j}).$$
\end{thm}
\begin{proof}
Let $R$ be an indecomposable right artinian right nonsingular right weakly CS right $a$-ring. We first show that $eR$ is quasi-injective for any idempotent $e\in R$. Since $R$ is right artinian, we have $\Soc(eR)\ne 0$. As $R$ is right automorphism-invariant and right weak CS, $eR$ is also a weak CS module. Therefore $\Soc(eR)$ is a simple module which is essential in $eR$, and so $eR$ is uniform. Therefore $eR$ is quasi-injective. Now rest of the proof follows from Theorem 23 in \cite{Jain2}. For the sake of completeness, we give the argument below.

Choose an independent family $\mathcal{F}=\{e_{i}R$
: $1\leq i\leq n\}$ of indecomposable right ideals such that $R=\oplus
_{i=1}^{n}e_{i}R$. After renumbering, we may write $R=[e_{1}R]\oplus \lbrack
e_{2}R]\oplus \cdots \oplus \lbrack e_{k}R]$, where for $1\leq i\leq k$, $%
[e_{i}R]$ denotes the direct sum of those $e_{j}R$ that are isomorphic to $%
e_{i}R$. Let $[e_{i}R]$ be a direct sum of $n_{i}$ copies of $e_{i}R$.
Consider $1\leq i<j\leq k$. We arrange the summands $[e_{i}R]$ in such a way that $%
l(e_{j}R)\leq l(e_{i}R)$. Suppose $e_{j}Re_{i}\neq 0$. Then we
have an embedding of $e_{i}R$ into $e_{j}R$, hence $l(e_{i}R)\leq
l(e_{j}R)$. But by assumption $l(e_{j}R)\leq l(e_{i}R)$, so $%
l(e_{i}R)=l(e_{j}R)$, we get $e_{j}R\cong e_{i}R$, which is
a contradiction. Hence $e_{j}Re_{i}=0$ for $j>i$. Thus we have
\[
R\cong \left[
\begin{array}{cccccc}
\mathbb{M}_{n_{1}}(e_{1}Re_{1}) & \mathbb{M}_{n_{1}\times n_{2}}(e_{1}Re_{2})
& . & . & . & \mathbb{M}_{n_{1}\times n_{k}}(e_{1}Re_{k}) \\
0 & \mathbb{M}_{n_{2}}(e_{2}Re_{2}) & . & . & . & \mathbb{M}_{n_{2}\times
n_{k}}(e_{2}Re_{k}) \\
0 & 0 & \mathbb{M}_{n_{3}}(e_{3}Re_{3}) & . & . & \mathbb{M}_{n_{3}\times
n_{k}}(e_{3}Re_{k}) \\
. & . & . & . & . & . \\
. & . & . & . & . & . \\
0 & 0 & . & . & . & \mathbb{M}_{n_{k}}(e_{k}Re_{k})
\end{array}
\right],
\]
where each $e_iRe_i$ is  a division ring, $e_iRe_i\simeq e_jRe_j$ for each $1\leq i,j\leq k$ and $n_1, \ldots, n_k $ are any positive integers.  Furthermore, if $e_iRe_j\ne 0$, then $$dim(_{e_iRe_i}(e_iRe_j))=1=dim((e_iRe_j)_{e_jRe_j}).$$ 
\end{proof}

\bigskip

\bigskip

{\bf Acknowledgement.} We would like to thank the referee for carefully reading the paper. The suggestions of the referee have improved the presentation of this paper. The second author has been partially funded by the Vietnam National Foundation for Science and Technology Development (NAFOSTED). Ashish K. Srivastava gratefully acknowledges the support from TUBITAK of Turkey for his visit to the Gebze Technical University, and the hospitality from the host university.

\bigskip

\bigskip

\bigskip

\bigskip

\begin{thebibliography}{99}

\bibitem{AFT} A. Alahmadi, A. Facchini, N. K. Tung, Automorphism-invariant modules, Rend. Sem. Mat. Univ. Padova, 133 (2015), 241-259.
\bibitem{AF}  F. W.  Anderson, K. R. Fuller, Rings and Categories of Modules, Second Edition, Springer, New York (1992).
\bibitem{BJ} K. I. Beidar, S. K. Jain, The structure of right continuous right $\pi$-rings, Comm. Algebra 32, 1 (2004), 315-332.
\bibitem{BJ1} K. I. Beidar, Y. Fong, W.-F. Ke, S. K. Jain, An example of right $q$-rings, Israel J. Math. 127 (2002), 303-316.
\bibitem{BS} W. D. Burgess, W. Stephenson, Rings all of whose Pierce stalks are local, Canad. Math. Bull. 22 (1979), 159-164.
\bibitem{Byrd} K. A. Byrd, Right self-injective rings whose essential ideals are two-sided, Pacific J. Math. 82 (1979), 23-41.
\bibitem{CH} J. Clark, D. V. Huynh, Simple rings with injectivity conditions on one-sided ideals, Bull. Australian. Math. Soc. 76 (2007), 315-320.
\bibitem{DF} S. E. Dickson, K. R. Fuller, Algebras for which every indecomposable right module is invariant in its injective envelope, Pacific J. Math. 31, 3 (1969), 655-658.
\bibitem{DHSW} N. V. Dung, D. V.  Huynh, P. F. Smith,  R. Wisbauer, Extending Modules, Pitman Research Notes in Math. 313, Longman (1994).
\bibitem{ESS} N. Er, S. Singh, A. K. Srivastava, Rings and modules which are stable under automorphisms of their injective hulls, J. Algebra 379 (2013) 223-229.
\bibitem{Er} N. Er, Direct sums and summands of weak CS-modules and continuous modules, Rocky Mountain J. Math. 29, 2 (1999), 491-503.
\bibitem{G2} K. R. Goodearl, Ring Theory, Nonsingular Rings and Modules, Monographs on Pure and Applied Mathematics Vol. 33. Dekker, New York, (1976).
\bibitem{GKS} P. A. Guil Asensio, D. Keskin T\"ut\"unc\"u, A. K. Srivastava, Modules invariant under automorphisms of their covers and envelopes, Israel J. Math. 206, 1 (2015), 457-482.
\bibitem{GS}  P. A. Guil Asensio, A. K. Srivastava, Automorphism-invariant modules satisfy the exchange property,  J. Algebra 388 (2013), 101-106.
\bibitem {AS1} P. A. Guil Asensio, A. K. Srivastava, Additive unit representations in endomorphism rings and an extension of a result of Dickson and Fuller, Ring Theory and Its Applications, Contemp. Math., Amer. Math. Soc., 609 (2014), 117-121.
\bibitem{AS2}  P. A. Guil Asensio, A. K. Srivastava, Automorphism-invariant modules, Noncommutative rings and their applications, Contemp. Math., Amer. Math. Soc. 634 (2015), 19-30.
\bibitem{Hill} D. A. Hill, Semiperfect $q$-rings, Math. Annalen 200 (1973), 113-121.
\bibitem{I}  G. Ivanov, Non-local rings whose ideals are quasi-injective, Bull. Australian Math. Soc. 6 (1972), 45-52.
\bibitem{I1} G. Ivanov, Non-local rings whose ideals are quasi-injective: Addendum, Bull. Australian Math. Soc. 12 (1975), 159-160.
\bibitem{I2} G. Ivanov,  On a generalization of injective von Neumann rings,  Proc. Amer. Math. Soc.  124 (1996), 1051-1060.
\bibitem{Jain} S. K. Jain, Rings whose cyclic modules have certain properties and the duals, in Ring Theory, Vol. 25, Proceedings of the Ohio University Conference, Marcel Dekker, (1977).
\bibitem{Jain1} S. K. Jain, S. R. L\'opez-Permouth, R. Syed, Rings with quasi-continuous right ideals, Glasgow Math. Journal 41 (1999), 167-181.
\bibitem{JMS} S. K. Jain, S. Mohamed, S. Singh, Rings in which every right ideal is quasi-injective, Pacific J. Math. 31 (1969), 73-79.
\bibitem{JS} S. K. Jain, S. Singh, Quasi-injective and pseudo-injective modules, Canad. Math. Bull. 18 (1975), 359-366.
\bibitem{Jain2} S. K. Jain, S. Singh, A. K. Srivastava, On $\Sigma$-$q$-rings, J. Pure and Appl. Algebra 213, 6 (2009), 969-976.
\bibitem{JST} S. K. Jain, A. K. Srivastava, A. A. Tuganbaev, Cyclic Modules and the Structure of Rings, Oxford Mathematical Monographs, Oxford Univ. Press, (2012).
\bibitem{Koeh1} A. Koehler, Rings for which every cyclic module is quasi-projective, Math. Annalen  189 (1970), 407-419.
\bibitem{Koeh2} A. Koehler, Rings with quasi-injective cyclic modules, Quart. J. Math. Oxford  25 (1974), 51-55.
\bibitem{lam} T. Y. Lam, A first course in noncommutative rings, Second Edition, Graduate Texts in Mathematics, 131, Springer-Verlag, New York, (2001).
\bibitem{LD} T. Y. Lam, A. S. Dugas, Quasi-duo rings and stable range descent, J. Pure and Appl. Algebra 195 (2005), 243-259.
\bibitem {Lee-Zhou}  T. K. Lee, Y.  Zhou, Modules which are invariant under automorphisms of their injective hulls, J. Algebra and Appl. 12, 2 (2013), 9 pages.
\bibitem{Mohe1} S. H. Mohamed, Rings whose homomorphic images are $q$-rings, Pacific J. Math. 35 (1970), 720-735.
\bibitem{Mohe2} S. H. Mohamed, $q$-rings with chain conditions, J. London Math. Soc. 2 (1972), 455-460.
\bibitem{MM} S. H. Mohammed, B. J. M\" uller, Continuous and Discrete Modules, London Math. Soc. LN 147:  Cambridge Univ. Press. (1990).
\bibitem{Ni76} W. K. Nicholson, Semiregular modules and rings. Canad. J. Math. 28(5) (1976), 1105-1120.
\bibitem {NY1} W. K. Nicholson, M. F. Yousif, Quasi-Frobenius Rings, Cambridge Univ. Press. (2003).
\bibitem{QK} T. C. Quynh, M. T.  Ko\c{s}an, On automorphism-invariant modules, J. Algebra and Appl. 14, 5 (2015), 11 pages.
\bibitem {SS} S. Singh, A. K. Srivastava, Rings of Invariant Module Type and Automorphism-Invariant Modules, Ring Theory and Its Applications, Contemp. Math., Amer. Math. Soc. 609 (2014), 299-311.
\bibitem {Smith} P. F. Smith, CS-modules and weak CS-modules, Noncommutative ring theory, Lecture Notes in Math., 1448, Springer, Berlin (1990), 99-115.
\bibitem{Teply} M. L. Teply, Pseudo-injective modules which are not quasi-injective, Proc. Amer. Math. Soc. 49, 2 (1975), 305-310.
\bibitem {Ut} Y. Utumi, On continuous rings and self injective rings, Trans. Amer. Math. Soc. 118 (1965), 1-11.
\end{thebibliography}
\end{document}